\documentclass{amsart}
\usepackage{amssymb}
\usepackage[all]{xy}

\title{A characterization of amenability of group actions on $C^\ast$-algebras}
\author{Masayoshi Matsumura}

\address{Department of Mathematical Science, University of Tokyo, Komaba, 153-8914, Japan}
\email{masa@ms.u-tokyo.ac.jp}

\newtheorem*{thm}{Theorem}
\newtheorem*{lem}{Lemma}

\newtheorem*{dfnthm}{Definition/Theorem}
\newtheorem*{dfnprop}{Definition/Proposition}

\begin{document}

\begin{abstract}
We show that coincidence of the full and reduced crossed product $C^\ast$-algebras of a group action on a unital commutative $C^\ast$-algebra implies amenability of the action whenever the group is exact. This is a partial answer to a problem posed by C. Anantharaman-Delaroche in 1987.

\end{abstract}

\maketitle

\section{Introduction}

Anantharaman-Delaroche introduced in 1987 amenability of group actions on $C^\ast$-algebras and proved that amenability of actions implies coincidence of the full and reduced crossed product $C^\ast$-algebras. Moreover she proved that an action on a nuclear $C^\ast$-algebra is amenable if and only if its reduced crossed product is nuclear \cite{AD}. We show the converse of the first result assuming that the group is exact in the sense of E. Kirchberg and S. Wassermann \cite{KW} and the algebra is unital and commutative. 
Moreover, we can also prove a corresponding result for noncommutative $C^\ast$-algebras under a certain technical assumption. More precisely, we have the following theorem.

\begin{thm}
Let $\Gamma$ be a (discrete) exact group and $A$ be a unital $\Gamma$-$C^\ast$-algebra. Assume that one of the following conditions is satisfied.

\begin{itemize}
\item[{\rm (i)}] The full and reduced crossed products of $A$ by $\Gamma$ coincide and $A$ is commutative.
\item[{\rm (ii)}] The full and reduced crossed products of $A \otimes A^{\rm op}$ by $\Gamma$ coincide and $A$ is nuclear.
\end{itemize}

Then the action $\Gamma \curvearrowright A$ is amenable.
\end{thm}

The consequence is natural since if the algebra is the field of complex numbers with a trivial group action, then the full and reduced crossed products are the full and reduced group algebras, and in this case it was already proved by Hulanicki in 1966 \cite{Hu}. However, in our case, the method is completely different to the one of Hulanicki. We use two properties related with the property $C'$ of R. J. Archbold and C. J. K. Batty \cite{AB} and the weak expectation property (WEP) of E. C. Lance \cite{L}.

Moreover, the assumption is also natural since it is known that a group is exact if and only if it acts amenably on a unital commutative $C^\ast$-algebra \cite{GK} \cite{O}.

Finally, we remark that the commutative case of our theorem can also be interpreted as a statement of locally compact \'etale groupoids which also have associated full and reduced $C^\ast$-algebras and amenability (cf: \cite{ADR}).

\section{Preliminaries}
In this section, we recall basic definitions and facts.

For a (discrete) group $\Gamma$, $\Gamma$-$C^\ast$-algebra $(A, \alpha)$ is a $C^\ast$-algebra $A$ with an action $\alpha$ of $\Gamma$ by $\ast$-automorphisms. We always consider $A^{\ast\ast}$ with a natural $\Gamma$-action, also denoted by $\alpha$, which is the ultraweakly continuous extension of the original $\alpha$. We also give the opposite $A^{\rm op}$ of $A$ an associated action $\alpha^{\rm op}$ such that $\alpha^{\rm op}_s (a^{\rm op}) = \alpha_s (a)^{\rm op}$ for any $a^{\rm op} \in A^{\rm op}$ and any $s \in \Gamma$. If $(B, \beta)$ is another $\Gamma$-$C^\ast$-algebra, then the minimal tensor product $A \otimes B$ naturally has the diagonal action $\alpha \otimes \beta$ of $\Gamma$, that is, $(\alpha \otimes \beta)_s = \alpha_s \otimes \beta_s$ for any $s \in \Gamma$. For a family of $\Gamma$-$C^\ast$-algebras $\{ (A_i, \alpha_i) \}_i$, its direct product $\prod A_i$ has a $\Gamma$-action $(\prod \alpha_i)_s = \prod \alpha_{i, s}$. We will omit $\alpha$ in the notation if no confusion occurs.

For a $\Gamma$-$C^\ast$-algebra $(A, \alpha)$, its algebraic crossed product $A \rtimes_{\rm alg} \Gamma$ is a $\ast$-algebra which consists of finitely supported $A$-valued functions on $\Gamma$ with twisted convolution product and involution, that is, $(\sum a_s \delta_s)(\sum b_t\delta_t)$ is defined to be $\sum a_s \alpha_s b_t\delta_{st}$ and the involution sends $\sum a_s \delta_s$ to $\sum \alpha_{s^{-1}} a_s^\ast \delta_{s^{-1}}$, where $a\delta_s$ denotes a $A$-valued function on $\Gamma$ which sends $s$ to $a$ and others to $0$. We may regard $A$ as a subalgebra of $A \rtimes_{\rm alg} \Gamma$ by $a \mapsto a\delta_e$ and $\Gamma$ as a subset of it by $s \mapsto 1 \delta_s$ if $A$ is unital.

The full crossed product of $A$ is the universal enveloping $C^\ast$-algebra of $A \rtimes_{\rm alg} \Gamma$. To define the reduced crossed product of $A$, assume that $A$ is embedded into some $B(H)$ with a unitary representation $u$ of $\Gamma$ on $H$ which implements the $\Gamma$-action on $A$, that is, $\alpha_s a = u_s a u_s^\ast$ for any $a \in A \subseteq B(H)$ and $s \in \Gamma$. Such an embedding always exists. Then the reduced crossed product of $A$ is defined to be the closure of $A \rtimes_{\rm alg} \Gamma \subseteq B(\ell^2 \Gamma \otimes H)$, where $a\delta_s$ acts on $\ell^2 \Gamma \otimes H$ as $a\delta_s (\xi \otimes \eta) = \lambda_s \xi \otimes au_s \eta$. The left regular representation $\Gamma \curvearrowright \ell^2 \Gamma$ is denoted as $\lambda$, that is, $(\lambda_s \xi)(t) = \xi (s^{-1}t)$ for any $s, t \in \Gamma$ and any $\xi \in \ell^2 \Gamma$. In fact, the definition does not depend on the choice of the embedding $A \subseteq B(H)$. 

The following is a central concept in this paper.

\begin{dfnthm}[\cite{AD}] A group action of $\Gamma$ on a $C^\ast$-algebra $A$ is said to be amenable if there exist a net of finitely supported positive definite functions $h_i : \Gamma \rightarrow Z(A^{\ast\ast})$ such that $h_i (e) \leq 1$ for any $i$ and $h_i (s)$ converges to $1$ ultraweakly for each $s \in \Gamma$.
If $A$ is nuclear as a $C^\ast$-algebra, then it is equivalent to saying that the reduced crossed product $A \rtimes_r \Gamma$ is nuclear. 
\end{dfnthm}

Here a map $h : \Gamma \rightarrow Z(A^{\ast\ast})$ is said to be positive definite if a matrix $[\alpha_{s_i} h(s_i^{-1} s_j)]_{ij}$ is positive for each finitely many elements $s_1,...,s_n$ in $\Gamma$.

The most important example of an amenable action in this paper comes from an exact group.

\begin{dfnthm}[\cite{KW}\cite{GK}\cite{O}] A group $\Gamma$ is said to be exact if for any exact sequence of $\Gamma$-$C^\ast$-algebras
\[ 0 \rightarrow J \rightarrow A \rightarrow A/J \rightarrow 0 \]
the associated sequence of $C^\ast$-algebras
\[ 0 \rightarrow J \rtimes_r \Gamma \rightarrow A \rtimes_r \Gamma \rightarrow A/J \rtimes_r \Gamma \rightarrow 0 \]
is also exact.

It is equivalent to saying that the left multiplication action $\Gamma \curvearrowright \ell^\infty \Gamma$ is amenable.
\end{dfnthm}

In this case, we have a net of positive definite maps in the definition of the amenable action with ranges contained in $\ell^\infty \Gamma$, not only its double dual (see Th\'eor\`eme 4.9 of \cite{AD}). Hence we can conclude that the diagonal action $\Gamma \curvearrowright \ell^\infty \Gamma \otimes A$ is also amenable for any unital $\Gamma$-$C^\ast$-algebras $A$.

The next ingredient is the Haagerup standard form, but we state only a part of the result which is sufficient for our purpose.

\begin{thm}[\cite{Ha}] For any $\Gamma$-$C^\ast$-algebra $(A, \alpha)$, there exist a faithful normal representation of $A^{\ast\ast}$ into $B(H)$, a conjugate linear isometric involution $J$ on $H$ and a unitary representation $u$ of $\Gamma$ on $H$ satisfying the following conditions.
\begin{itemize}
\item[{\rm (i)}] $J A^{\ast\ast} J = (A^{\ast\ast})'$.
\item[{\rm (ii)}] $J z J = z^\ast$ for any element $z$ in the center of $A^{\ast\ast}$.
\item[{\rm (iii)}] $J = u_s J u_s^\ast$ for any $s \in \Gamma$.
\item[{\rm (iv)}] $\alpha_s a = u_s a u_s^\ast$ for any $a \in A$ and $s \in \Gamma$.
\end{itemize}
\end{thm}

We observe that if $A$ is commutative, then $A^{\ast\ast}$ is maximally abelian in the standard form by (i) and (ii) of the theorem. We also see that $a^{\rm op} \mapsto J a^\ast J$ defines a universal representation of the opposite $A^{\rm op}$ of $A$. The weak closure of its range is the commutant of $A^{\ast\ast}$ by (i) and $u$ implements the naturally induced $\Gamma$-action on $A^{\rm op}$ by (iii) and (iv).

Finally, we recall multiplicative domains of completely positive (cp) maps (see, for example, Proposition 1.5.7 of \cite{BO}).

\begin{dfnprop} For a cp map $\varphi : A \rightarrow B$, its multiplicative domain $A_{\varphi}$ is defined as $A_{\varphi} = \{ a \in A : \varphi (aa^\ast) = \varphi (a) \varphi (a)^\ast \ {\rm and} \ \varphi (a^\ast a) = \varphi (a)^\ast \varphi (a) \}$. Then $\varphi$ satisfies $\varphi (ab) = \varphi (a) \varphi (b)$ and $\varphi (ba) = \varphi (b) \varphi (a)$ for any $a \in A_{\varphi}$ and $b \in A$. 
\end{dfnprop}

We will use this Proposition for cp maps which is an extension of a $\ast$-homomorphism. Then the domain of $\ast$-homomorphism is contained in the multiplicative domain of the cp map.

\section{Proof of the main theorem}

Let us start to prove the theorem. We need two lemmata to prove the theorem. These are related with the property $C'$ of Archbold and Batty and the WEP of Lance as already mentioned in the introduction.

The proof and the consequence of the first lemma is similar to the one of that exactness implies property $C'$ for $C^\ast$-algebras (cf: Prop 9.2.7 of \cite{BO}).
\begin{lem}
For any exact group $\Gamma$ and any $\Gamma$-$C^\ast$-algebra $A$, the natural map $A^{\ast\ast} \rtimes_{\rm alg} \Gamma \rightarrow (A \rtimes_r \Gamma)^{\ast\ast}$ extends to a $\ast$-homomorphism on $A^{\ast\ast} \rtimes_r \Gamma$.
\end{lem}
\begin{proof}
First we remark that the natural inclusion $A \hookrightarrow A \rtimes_r \Gamma \hookrightarrow (A \rtimes_r \Gamma)^{\ast\ast}$ extends to $A^{\ast\ast} \rightarrow (A \rtimes_r \Gamma)^{\ast\ast}$ which induces a map $A^{\ast\ast} \rtimes_{\rm alg} \Gamma \rightarrow (A \rtimes_r \Gamma)^{\ast\ast}$ from an algebraic crossed product.
Let 
$$A_I = \{(x_i)_i \in \prod_{i \in I} A_i : (x_i)_i {\rm \ is \ strong^\ast \mathchar`-convergent \ in } \ A^{\ast\ast} \}$$
 for a directed set $I$ where $A_i$ is a copy of $A$. This is a $\Gamma$-invariant $C^{\ast}$-subalgebra of $\prod_{i \in I} A_i$. The $\ast$-homomorphism $\pi : A_I \rightarrow A^{\ast\ast}$, assigning a net to its strong$^\ast$-limit, is surjective for a sufficiently large $I$ by Kaplansky's density theorem. Fix such a set $I$. Then the natural map $\tilde{\pi} : A_I \rtimes_{\rm alg} \Gamma \subseteq A^{\ast\ast} \rtimes_{\rm alg} \Gamma \rightarrow (A \rtimes_r \Gamma)^{\ast\ast}$ is continuous in the topology of  $A_I \rtimes_r \Gamma$ since for any $\sum_{k = 1}^n (x_i^k)_i  \delta_{s_k} \in A_I \rtimes_{\rm alg} \Gamma$ we have

\begin{align*}
\| \tilde{\pi} ( \sum_{k = 1}^n (x_i^k)_i \delta_{s_k} ) \| &= \| {\rm strong^\ast \mathchar`-}\lim_i \sum_{k = 1}^n x_i^k \delta_{s_k} \| \\
&\leq \sup_i \| \sum_{k = 1}^n x_i^k \delta_{s_k} \| \\
&= \| \sum_{k = 1}^n (x_i^k)_i \delta_{s_k} \|. 
\end{align*}

The last equality follows from the fact that there exists a natural embedding $(\prod_{i \in I} A_i) \rtimes_r \Gamma \hookrightarrow \prod_{i \in I} (A_i \rtimes_r \Gamma)$, $\sum_k (x_i^k)_i \delta_{s_k} \mapsto (\sum_k x_i^k \delta_{s_k})_i$. 
This is well-defined and injective since if $A_i \subseteq B(H_i)$, then $(\prod_{i \in I} A_i) \rtimes_r \Gamma \subseteq B(\ell^2 \Gamma \otimes (\oplus H_i))$ and $\prod_{i \in I} (A_i \rtimes_r \Gamma) \subseteq B(\oplus (\ell^2 \Gamma \otimes H_i))$ and an obvious unitary implements the embedding.
Thus we have a $\ast$-homomorphism $A_I \rtimes_r \Gamma \rightarrow (A \rtimes_r \Gamma)^{\ast\ast}$ which vanishes on $({\rm ker} \pi) \rtimes_r \Gamma$. Since we have assumed that $\Gamma$ is exact, we have $A^{\ast\ast} \rtimes_r \Gamma \simeq (A_I \rtimes_r \Gamma)/(({\rm ker} \pi) \rtimes_r \Gamma) \rightarrow (A \rtimes_r \Gamma)^{\ast\ast}$.
\end{proof}

The second lemma corresponds in the non-equivariant case to a statement that nuclearity implies the WEP. In the noncommutative case, we need a possibly stronger assumption.

\begin{lem}
Let $(A, \alpha)$ be a unital $\Gamma$-$C^\ast$-algebra. Assume that one of the following conditions is satisfied.

\begin{itemize}
\item[{\rm (i)}] The full and reduced crossed products of $A$ by $\Gamma$ coincide and $A$ is commutative.
\item[{\rm (ii)}] The full and reduced crossed products of $A \otimes A^{\rm op}$ by $\Gamma$ coincide and $A$ is nuclear. Here $A^{\rm op}$ is the opposite of $A$.
\end{itemize}

Then there exists a $\Gamma$-equivariant cp map $\ell^\infty \Gamma \otimes A \rightarrow A^{\ast\ast}$ which is the identity on $A$, where $A$ is considered to be embedded into $\ell^\infty \Gamma \otimes A$ as $\mathbb{C} 1_{\ell^\infty \Gamma}  \otimes A$
\end{lem}

\begin{proof}
Let $A^{\ast\ast} \subseteq B(H)$ be the Haagerup standard form of $A^{\ast\ast}$. Note that $(A^{\rm op})^{\ast\ast}$ has a natural embedding into $B(H)$ such that $(A^{\rm op})^{\ast\ast} = (A^{\ast\ast})'$ and the $\Gamma$-actions on $A$ and $A^{\rm op}$ are implemented by a unitary representation $u$ of $\Gamma$ on $H$ as explained in the previous section. 
\begin{itemize}
\item[(i)] First, assume that the full and reduced crossed products of $A$ by $\Gamma$ coincide and $A$ is commutative. Consider a representation $\pi : A \rtimes_r \Gamma = A \rtimes \Gamma \rightarrow B(H)$ induced by the covariant representation. 
On the other hand, $A \rtimes_r \Gamma$ is embedded in $B(\ell^2 \Gamma \otimes H)$ so that $a\delta_s$ acts as $\lambda_s \otimes a u_s$ for any $a \in A$ and any $s \in \Gamma$.
By Arveson's extension theorem we have a cp map $\tilde{\pi} : B(\ell^2 \Gamma \otimes H) \rightarrow B(H)$ which restricts to $\pi$ on $A \rtimes_r \Gamma$.
Then the restriction of the cp map to $\ell^\infty \Gamma \otimes A \subseteq B(\ell^2 \Gamma \otimes H)$, where elements of $\ell^\infty \Gamma$ act on $\ell^2 \Gamma$ as multiplication operators, does the work for the following reason. First, note that any $x \in \ell^\infty \Gamma \otimes A$ commutes with $1 \otimes a \in 1 \otimes A$ and the latter is in the multiplicative domain of $\tilde{\pi}$. Hence we have $a \tilde{\pi}(x) = \tilde{\pi}((1 \otimes a) x) = \tilde{\pi}(x (1 \otimes a)) = \tilde{\pi}(x) a$, that is, $\tilde{\pi}( \ell^\infty \Gamma \otimes A) \subseteq A' = A^{\ast\ast}$ since $A$ is commutative. To show that the cp map $\tilde{\pi} : \ell^\infty \Gamma \otimes A \rightarrow A$ is $\Gamma$-equivariant, we only need to know that, by unitarity of $A$, $\lambda_s \otimes u_s \in B(\ell^2 \Gamma \otimes H)_{\tilde{\pi}}$, which is sent to $u_s$ by $\tilde{\pi}$ and $\lambda$ implements the left multiplication action $\Gamma \curvearrowright \ell^\infty \Gamma$. Indeed, we have $\tilde{\pi}(s(f \otimes a)) = \tilde{\pi}((\lambda_s \otimes u_s)(f \otimes a)(\lambda_s \otimes u_s)^\ast) = u_s \tilde{\pi}(f \otimes a) u_s^\ast = \alpha_s \tilde{\pi}(f \otimes a)$ for any $f \in \ell^\infty \Gamma$, $a \in A$ and $s \in \Gamma$.
\item[(ii)] In this case, consider a natural representation $\pi : (A \otimes A^{\rm op}) \rtimes_r \Gamma = (A \otimes A^{\rm op}) \rtimes \Gamma \rightarrow B(H)$, $(a \otimes b^{\rm op})\delta_s \mapsto a J b^\ast J u_s$, induced by the Haagerup standard form of $A$. This extends to a cp map $\tilde{\pi} : ((\ell^\infty \Gamma \otimes A) \otimes A^{\rm op}) \rtimes_r \Gamma \rightarrow B(H)$ whose restriction to $\ell^\infty \Gamma \otimes A \otimes \mathbb{C} 1_{A^{\rm op}} \simeq \ell^\infty \Gamma \otimes A$ gives us a desired $\Gamma$-equivariant cp map, using the similar ``multiplicative domian argument''. Namely, apply it twice to reduce its range and to show that it is $\Gamma$-equivariant. Reduction follows from the fact that $1 \otimes 1 \otimes a^{\rm op} \in \ell^\infty \Gamma \otimes A \otimes A^{\rm op} \subseteq (\ell^\infty \Gamma \otimes A \otimes A^{\rm op}) \rtimes_r \Gamma$ is in the multiplicative domain. To show that the cp map is equivariant, we see that $s \in \Gamma \subseteq (\ell^\infty \Gamma \otimes A \otimes A^{\rm op}) \rtimes_r \Gamma$ is in the multiplicative domain. 
\end{itemize}

\end{proof}

Now, we recall and prove the main theorem.

\begin{thm}
Assume that $A$ is a $\Gamma$-$C^\ast$-algebra satisfying the conditions of the last lemma. Moreover, assume that $\Gamma$ is exact. Then $A \rtimes_r \Gamma$ is nuclear. In other words the action $\Gamma \curvearrowright A$ is amenable.
\end{thm}

\begin{proof}
Combining the previous lemmata and considering the functoriality of taking reduced crossed products with respect to equivariant cp maps, we have the following commutative diagram:

\[\xymatrix{
(\ell^\infty \Gamma \otimes A) \rtimes_r \Gamma \ar[r] & A^{\ast\ast} \rtimes_r \Gamma \ar[r] & (A \rtimes_r \Gamma)^{\ast\ast} \\
A \rtimes_r \Gamma \ar[u] \ar[ru] \ar[rru]
}\]

Thus the natural inclusion $A \rtimes_r \Gamma \hookrightarrow (A \rtimes_r \Gamma)^{\ast\ast}$ factors through $(\ell^\infty \Gamma \otimes A) \rtimes_r \Gamma$ which is nuclear since the action $\Gamma \curvearrowright \ell^\infty \Gamma \otimes A$ is amenable whenever $\Gamma$ is exact and $A$ is unital. Therefore $A \rtimes_r \Gamma$ is nuclear. Recall that any $C^\ast$-algebra whose inclusion into its double dual is nuclear is nuclear (see Proposition 2.3.8. of \cite{BO}).
\end{proof}

Finally we remark that a possible counterexample is the non-exact group $\Gamma$ constructed by Gromov \cite{G} (see also \cite{O}). But we do not know whether $\ell^\infty \Gamma \rtimes \Gamma = \ell^\infty \Gamma \rtimes_r \Gamma$ or not. 

\section*{acknowledgement} I would like to thank Y. Kawahigashi and N. Ozawa for valuable discussions and comments.

\end{document}